\documentclass[12pt,a4paper]{amsart}

\usepackage{latexsym} 
 
\usepackage{graphicx}
\usepackage{epsfig}
\usepackage{mathrsfs}
\usepackage{amssymb}
\usepackage{amsthm}
\usepackage{amsfonts}
\usepackage{amsmath}
\usepackage{centernot}
\usepackage{amstext}
\usepackage{amscd}
\usepackage{enumerate}
\usepackage{tikz}
\usepackage{epic}
\usepackage{eepic}
\usepackage{pstricks,pst-plot}
\usepackage{bigdelim}
\usepackage{multirow}

\numberwithin{equation}{section}

\theoremstyle{plain}
\newtheorem{thm}{Theorem}[section] 
\newtheorem{prop}[thm]{Proposition}
\newtheorem{cor}[thm]{Corollary}
\newtheorem{lem}[thm]{Lemma}

\newtheorem{theorem*}{Theorem}[]

\theoremstyle{definition}
\newtheorem{defn}[thm]{Definition}

\theoremstyle{remark}
\newtheorem{rmk}[thm]{Remark}

\theoremstyle{property}

\newcommand{\N}{\mathbb{N}}
\newcommand{\R}{\mathbb{R}}

\DeclareMathOperator{\grad}{grad\,}

\def\accentclass@{7}
\def\makeacc@#1#2{\def#1{\mathaccent"\accentclass@#2 }}
\makeacc@\cir{017}

\newcommand{\s}{\Sigma}

\def\dis{\displaystyle}

\def\inf{\mathop{\rm inf}}
\def\det{\mathop{\rm det}}

\def\min{\mathop{\rm min}}

\newcommand{\de}{\delta}
\newcommand{\sing}{\mathrm{Sing}}

\def\det{{\text {\rm det}}}

\def\a {{\alpha}}

\def\d {{\delta}}
\def\D {{\Delta}}

\title[relative Kuo condition]
{On the relative Kuo condition \\
and the second relative Kuo condition}

\author{Karim Bekka and Satoshi Koike} 

\address{Institut de recherche Mathematique de Rennes, 
Universit\'{e} de Rennes 1, Campus Beaulieu, 35042 Rennes cedex, France}
\address{Department of Mathematics, Hyogo University of Teacher Education,
Kato, Hyogo 673-1494, Japan}
\email{karim.bekka@univ-rennes1.fr}
\email{koike@hyogo-u.ac.jp} 

\date{\today}

\begin{document}

\thanks{This research is partially supported by the Grant-in-Aid 
for Scientific Research (No. 26287011) of Ministry of Education, 
Science and Culture of Japan, and HUTE Short-Term Fellowship
Program 2016.}


\maketitle

\begin{abstract}
The Kuo condition and the the second Kuo condition are known as criteria for an 
$r$-jet to be $V$-sufficient in $C^r$ mappings and $C^{r+1}$ mappings, 
respectively.
In \cite{bekkakoike2} we considered the notions of 
$V$-sufficiency of jets and these Kuo conditions in the relative case 
to a given closed set $\s$, and showed that the relative Kuo condition 
is a criterion for a relative $r$-jet to be $\s$-$V$-sufficient in $C^{r}$ 
mappings and the second relative Kuo condition is a sufficient condition for 
a relative $r$-jet to be $\s$-$V$-sufficient in $C^{r+1}$ mappings. 
In this paper we discuss several conditions equivalent to the relative Kuo 
condition or the second relative Kuo condition.
\end{abstract}

\bigskip

\section{Introduction}
Tzee-Char Kuo formulated in \cite{kuo3} necessary and sufficient conditions, 
called the Kuo condition and the second Kuo condition, for an $r$-jet to be 
$V$-sufficient in $C^r$ mappings and $C^{r+1}$ mappings, respectively.
Sufficiency of jets is a notion introduced by Ren\'e Thom related to 
the structural stability problem.
$V$-sufficiency is sufficiency on the topology of the zero-set 
of a mapping.

Let $\s$ be a germ of a closed subset of $\R^n$ at $0 \in \R^n$ 
such that $0 \in \s$. 
In \cite{bekkakoike2} we generalised the notions of $V$-sufficiency of jets 
and the Kuo conditions to the relative case to $\s$ . 
Then we proved that the relative Kuo condition is a criterion for a relative 
$r$-jet to be $\s$-$V$-sufficient in $C^r$ mappings (see \S \ref{relativekuo}),
and that the second relative Kuo condition is a sufficient condition 
for a relative $r$-jet to be $\s$-$V$-sufficient in $C^{r+1}$ mappings 
(see \S \ref{2ndrelativekuo}). 

In \S \ref{mainresults} we show the main results of this paper on equivalent 
conditions to the above two relative Kuo conditions. 
In \S \ref{equivrel} we show a result on conditions equivalent to 
the relative Kuo condition (Theorem \ref{equiv1}).
In \S \ref{equiv2ndrel} we show two kinds of results on conditions equivalent 
to the second relative Kuo condition (Theorems \ref{prop223}, \ref{prop226}).

We give the definition of $\s$-$V$-sufficiency of jet, namely $V$-sufficiency 
of the relative jet to $\s$ in \S \ref{Preliminaries}, and the definitions of 
the relative Kuo condition and the second relative Kuo condition in 
\S \ref{criteria}.


Throughout this paper, let us denote by $\mathbb{N}$ the set of natural 
numbers in the sense of positive integers.

\bigskip


\section{Preliminaries}\label{Preliminaries}

   
Let $s \in \N \cup \{ \infty , \omega\}$. 
Let ${\mathcal E}_{[s]}(n,p)$ denote the set of
$C^s$ map-germs : $(\R^n,0)\to (\R^p,0)$, let $j^r f(0)$ denote the r-jet of 
$f$ at $0 \in \R^n$ for $f \in {\mathcal E}_{[s]}(n,p)$, $s \ge r$,
and let $J^r(n,p)$ denote the set of r-jets in ${\mathcal E}_{[s]}(n,p)$.

Throughout this paper, let $\s$ denote a germ of a closed subset of $ \R^n$ 
at $0 \in \R^n$ such that $0 \in \s.$  
Then we denote by $d(x,\s)$ the distance from a point $x \in \R^n$ 
to the subset $\s.$
 
We consider on ${\mathcal E}_{[s]}(n,p)$ the following equivalence relation: 

\vspace{1mm}

\noindent Two map-germs  $f,g\,\in {\mathcal E}_{[s]}(n,p)$ are 
$r$-$\Sigma$-{\em equivalent}, denoted by $f\sim g$, if there exists 
a neighbourhood $U$ of $0$ in $\R^n$ such that the r-jet extensions of 
$f$ and $g$ satisfy $j^rf(\s\cap U)= j^rg(\s\cap U).$

\vspace{1mm}

\noindent We denote by $j^rf(\s;0)$ the equivalence 
class of $f$, and by $J^r_{\s}(n,p)$ the quotient set 
${\mathcal E}_{[s]}(n,p)/\sim.$ 

\begin{rmk}\label{realisation}
Any $r$-jet, $r \in \N$, has a unique polynomial realisation 
of degree not exceeding $r$ in the non-relative case.
But relative-jets do not always have a $C^{\omega}$ realisation.
See Remark 2.1 in \cite{bekkakoike2} for the details.
\end{rmk}

Let us introduce an equivalence for elements of ${\mathcal E}_{[s]}(n,p)$.
We say that  $f,g\,\in {\mathcal E}_{[s]}(n,p)$ are 
$\s$-V-{\em equivalent}, if $f^{-1}(0)$ is homeomorphic to $g^{-1}(0)$
as germs at $0\in \mathbb{R}^n$ by a homeomorphism which fixes 
$f^{-1}(0)\cap \s$.

Let $w \in J^r_{\s}(n,p).$ 
We call the relative jet $w$ $\s$-V-{\em sufficient} in 
${\mathcal E}_{[s]}(n,p)$, $s \geq r$, if any two realisations
$f$, $g\,\in {\mathcal E}_{[s]}(n,p)$ of $w,$ 
namely $j^rf(\s;0) = j^rg(\s;0)=w,$  are $\s$-V-equivalent.

We next prepare notations concerning large and small relation and 
equivalence between two non-negative functions.

\begin{defn}
Let  $f,g : U \to \R$ be non-negative functions,
where $U \subset \R^N$ is an open neighbourhood of $0 \in \R^N$.
If there are real numbers $K > 0$, $\delta > 0$ 
with $B_{\delta}(0) \subset U$ such that
$$
f(x ) \le K g(x ) \ \ \text{for any} \ \ x \in B_{\delta}(0),
$$
where $B_{\delta}(0)$ is a closed ball in $\R^N$ of radius $\delta$
centred at $0 \in \R^N$, 
then we write $f \precsim g$ (or $g \succsim f$).
If $f \precsim g$ and $f \succsim g$, we write $f \thickapprox g$.
\end{defn}

At the end of this section, let us recall a useful lemma to treat 
the $C^{r+1}$ case.

\begin{lem}\label{lemrflat}(\cite{bekkakoike2})
Let $r \in \N$, and let $\s$ be a germ at $0\in \mathbb{R}^n$ of 
a closed set. 
Let $f: ( \mathbb{R}^n,0)\to (\mathbb{R}^p,0)$ be a $C^r$ map-germ, 
$r\geq 1,$ such that $j^rf(\s;0)=\{0\}.$ 
Then
$
\|f(x)\|=o(d(x,\s)^r).
$
If  moreover $ f$ is of classe $ C^{r+1}$, then 
$
\|f(x)\| \precsim d(x,\s)^{r+1}.
$
\end{lem} 

\bigskip

\section{Criteria for $\s$-$V$-sufficiency of jets}\label{criteria}

\subsection{The relative Kuo condition}\label{relativekuo}
We suppose now on the germ $\s$ fixed, and introduce the relative 
notion to $\s$ of the Kuo condition.
A criterion for an $r$-jet to be $C^0$-sufficient or $V$-sufficient 
in $C^r$ functions (resp. in $C^{r+1}$ functions) is known as 
the Kuiper-Kuo condition (resp. the second Kuiper-Kuo condition) 
in the non-relative, function case (see N. Kuiper \cite{kuiper}, T.-C. Kuo 
\cite{kuo1} and J. Bochnak - S. Lojasiewicz \cite{bochnaklojasiewicz}).
In this case the Kuiper-Kuo condition is equivalent to the Kuo condition 
(\cite{bochnaklojasiewicz}). 

Let $v_1, \cdots , v_p$ be $p$ vectors in $\R^n$ where $n \ge p$.
The {\em Kuo distance $\kappa$} (\cite{kuo3}) is defined by
$\dis
\kappa(v_1, \ldots,v_p) = \displaystyle \min_{i}\{\text{distance of }\, 
v_i\, \text{ to }\, V_i\},
$
where $V_i$ is the span of the $ v_j$'s, $j\ne i$.
In the case where $p = 1$, $\kappa (v) = \| v \| .$

We first recall the notion of the relative Kuo condition.
The original condition was introduced by T.-C. Kuo \cite{kuo3}
as a criterion of $V$-sufficiency of jets in the mapping case.

\begin{defn}[The relative  Kuo condition]\label{K}  
A map germ $f\in {\mathcal E}_{[r]}(n,p)$, $n\geq p$, satisfies the 
{\it relative  Kuo condition $(K_{\s})$} if there are strictly positive 
numbers $C, \alpha$ and $ \bar w$ such that
\begin{equation*}
\kappa(df(x))\geq Cd(x,\s)^{r-1} \text{ in } 
{\mathcal H}^{\s}_{r}(f; \bar w)\cap \{\| x \| < \alpha\},
\end{equation*} 
 $$\text{ namely, } \kappa(df(.))\succsim d(.,\s)^{r-1} \text{ on a set of points where }  \| f \| \precsim \ d(.,\s)^{r}.$$
\end{defn} 

In the definition \ref{K}, ${\mathcal H}^{\s}_{r}(f;\bar w)$ denotes the 
{\em horn-neighbourhood of $f^{-1}(0)$ relative to $\s$ of degree $r$ and 
width $\bar{w}$},
$$
{\mathcal H}^{\s}_{r}(f;\bar w) := \{x\in \mathbb{R}^n: \| f(x) \|
\leq\bar w\ d(x,\s)^{r}\}.
$$
The horn-neighbourhood was originally introduced in \cite{kuo2} 
in the non-relative case.

By definition, it is easy to see that the relative Kuo condition ($K_{\s}$) 
is equivalent to the following condition. 

\begin{defn}[Condition ($\widetilde{K}_{\s}$)]
A map germ $f\in {\mathcal E}_{[r]}(n,p)$, $n\geq p$, satisfies
{\it condition} ($\widetilde{K}_{\s}$) if 
\begin{equation*}
d(x,\s)\kappa(df(x))+\|f(x)\|\succsim d(x,\s)^{r} 
\end{equation*} 
holds in some neighbourhood of $0 \in \R^n.$
\end{defn}

\begin{rmk}\label{remark210}
\begin{enumerate}[1)]
\item Condition ($\widetilde{K}_{\s}$) was introduced in \cite{bekkakoike1}, 
in the non-relative case, namely $\s=\{0\}$.
\item The relative Kuo condition $(K_{\s})$ and condition 
($\widetilde{K}_{\s}$) are invariant under rotation.
\end{enumerate}
\end{rmk} 

As a sufficient condition for an $r$-jet to be $\s$-$V$-sufficient 
in ${\mathcal E}_{[r]}(n,p)$, we have the following result.

\begin{thm}\label{RelativeKuoThm}(\cite{bekkakoike2}) 
Let $r \in \N$, and let
$f \in {\mathcal E}_{[r]}(n,p)$, $n \ge p$.
If $f$ satisfies condition $(\widetilde{K}_{\s})$,
then the relative $r$-jet, $j^r f(\s;0)$, is $\s$-$V$-sufficient 
in ${\mathcal E}_{[r]}(n,p)$.
\end{thm}  

On the other hand, we have the following criteria for an $r$-jet to be 
$\s$-$V$-sufficient in ${\mathcal E}_{[r]}(n,p)$.

\begin{thm}\label{RelativeKuoThm1}(\cite{bekkakoike2}) 
Let $r \in \N$, and let 
$f \in {\mathcal E}_{[r]}(n,p)$, $n > p$. 
Then the following conditions are equivalent.
\begin{enumerate}[(1)]
\item $f$ satisfies the relative Kuo condition $(K_{\s})$.
\item $f$ satisfies condition $(\widetilde{K}_{\s})$.
\item The relative $r$-jet $j^r f(\s;0)$ is $\s$-$V$-sufficient 
in ${\mathcal E}_{[r]}(n,p)$.
\end{enumerate}
\end{thm}  

\begin{thm}\label{RelativeKuoThm2}(\cite{bekkakoike2}) 
Let $r \in \N$, and let 
$f \in {\mathcal E}_{[r]}(n,n)$.
Suppose that $j^rf(\s ;0)$ has a subanalytic $C^r$-realisation
and that $\s$ is a subanalytic closed subset of $\R^n$ 
such that $0 \in \s$. 
Then the following conditions are equivalent.
\begin{enumerate}[(1)]
\item $f$ satisfies the relative Kuo condition $(K_{\s})$.
\item $f$ satisfies condition $(\widetilde{K}_{\s})$.
\item The relative $r$-jet $j^r f(\s;0)$ is $\s$-$V$-sufficient 
in ${\mathcal E}_{[r]}(n,n)$.
\end{enumerate}
\end{thm}  

For the subanalyticity, see H. Hironaka \cite{hironaka}.


\subsection{The second relative Kuo condition}\label{2ndrelativekuo}

We next recall the notion of the second relative Kuo condition.
The original condition was introduced also by T.-C. Kuo \cite{kuo3} 
as a criterion of $V$-sufficiency of $r$-jets in $C^{r+1}$ mappings.

\begin{defn}[The second relative Kuo condition]\label{Kd}  
A map germ $f\in {\mathcal E}_{[r+1]}(n,p)$, $n\geq p$, satisfies the 
{\it second relative  Kuo condition $(K_{\s}^\de)$} if for any 
map $g\in {\mathcal E}_{[r+1]}(n,p)$ 
satisfying $j^{r}g(\s;0)=j^{r}f(\s;0)$ there are strictly positive numbers
$C, \alpha,\delta$ and $ \bar w$ (depending on $g$), such
that
\begin{equation*}
\kappa(df(x))\geq Cd(x,\s)^{r-\delta} \text{ in }  \mathcal{H}^{\s}_{r+1}(g;\bar w)\cap\{\|x\|<\a\},
\end{equation*} 
namely, $\kappa(df(.))\succsim d(.,\s)^{r-\delta} $
on a set of points where $\| g(.) \| \precsim \ d(.,\s)^{r+1}.$
\end{defn} 

\begin{rmk}\label{rmk2ndRelativeKuo}
\begin{enumerate}[1)]
\item For a  map $f\in {\mathcal E}_{[r]}(n,p)$ satisfying the  relative  Kuo condition or the second relative  Kuo condition, in a neighbourhood of $0 \in \R^n,$ the intersection of the singular set of $f,$ $\sing (f),$ and the horn 
neighbourhood
${\mathcal H}^{\s}_{r}(f; \bar w)$ is contained in $\s$, namely 
$$
\sing(f)\cap {\mathcal H}^{\s}_{r}(f; \bar w)\subset \s.
$$
In particular, in a neighbourhood of $0 \in \R^n$,  
$\mathrm{grad}f_{1}(x),\ldots,\mathrm{grad}f_{p}(x) $ are linearly independent  on $f^{-1}(0)\setminus \s.$ 

\item  For a map $f\in {\mathcal E}_{[r]}(n,p)$ satisfying the second relative 
Kuo, we have for any map 
$g\in {\mathcal E}_{[r+1]}(n,p)$ satisfying $j^{r}g(\s;0)=j^{r}f(\s;0)$, in a 
neighbourhood of $0 \in \R^n,$ the intersection of the singular set of $f,$ 
$\sing (f),$ and the horn neighbourhood
${\mathcal H}^{\s}_{r+1}(g; \bar w)$ is contained in $\s$, namely 
$
\sing(f)\cap {\mathcal H}^{\s}_{r+1}(g; \bar w)\subset \s.
$
Since   $\|(f-g)(x)\|\precsim d(x,\s)^{r+1},$ we have $f^{-1}(0)\subset {\mathcal H}^{\s}_{r+1}(g; \bar w) ,$
then, in a neighbourhood of $0 \in \R^n$,  
$\mathrm{grad}f_{1}(x),\ldots,\mathrm{grad}f_{p}(x) $ are linearly independent  on $f^{-1}(0)\setminus \s.$ 

\end{enumerate}

\end{rmk} 

We next consider a condition of type condition ($\widetilde{K}_{\s}$) 
in the $C^{r+1}$ case.

\begin{defn}[Condition ($\widetilde{K}^\d_{\s}$)]
A map germ $f\in {\mathcal E}_{[r+1]}(n,p)$, $n\geq p$, satisfies
{\it condition ($\widetilde{K}_{\s}^\d$})
if for any 
map $g\in {\mathcal E}_{[r+1]}(n,p)$ 
satisfying $j^{r}g(\s;0)=j^{r}f(\s;0)$ there exists 
$\delta>0$  (depending on $g$), such
that
\begin{equation*}
d(x,\s)\kappa(df(x))+\|g(x)\|\succsim d(x,\s)^{r+1-\d} 
\end{equation*} 
holds in some neighbourhood of $0 \in \R^n.$

\end{defn}

\begin{rmk}\label{remark211}
\begin{enumerate}[(1)]
\item The second relative Kuo condition ($K^\d_{\s}$), and condition 
($\widetilde{K}^\d_{\s}$) are invariant under rotation.
\item Condition ($\widetilde{K}^\d_{\s}$) can be equivalently written as:
for any map $g\in {\mathcal E}_{[r+1]}(n,p)$ satisfying 
$j^{r}g(\s;0)=j^{r}f(\s;0)$ there exists $\delta>0$  (depending on $g$), 
such
that
\begin{equation*}
d(x,\s)\kappa(dg(x))+\|g(x)\|\succsim d(x,\s)^{r+1-\d} 
\end{equation*} 
holds in some neighbourhood of $0 \in \R^n.$

\end{enumerate}

\end{rmk}

As a sufficient condition for an $r$-jet to be $\s$-$V$-sufficient 
in ${\mathcal E}_{[r+1]}(n,p)$, we have the following result.

\begin{thm}\label{RelativeKuoThm3}(\cite{bekkakoike2})
Let $r$ be a positive integer, and let
$f \in {\mathcal E}_{[r+1]}(n,p)$, $n \ge p$.
If $f$ satisfies condition $(K^\d_{\s})$,
then the relative $r$-jet, $j^r f(\s;0)$ is $\s$-$V$-sufficient 
in ${\mathcal E}_{[r+1]}(n,p)$.
\end{thm} 

As a corollary of Theorem \ref{RelativeKuoThm3} and Lemma \ref{lemrflat}, 
we have the following.

\begin{cor}\label{RelativeKuoCor2}(\cite{bekkakoike2})
Let $r$ be a positive integer, and let $f \in {\mathcal E}_{[r+1]}(n,p)$,
$n \ge p$.
If there exists $\delta>0$ such that
\begin{equation*}
d(x,\s)\kappa(df(x))+\|f(x)\|\succsim d(x,\s)^{r+1-\d}
\end{equation*} 
holds in some neighbourhood of $0 \in \R^n,$
then $j^rf(\s;0)$ is $\s$-$V$-sufficient in ${\mathcal E}_{[r+1]}(n,p).$
\end{cor}

\bigskip


\section{Main results}\label{mainresults}

\subsection{Conditions equivalent to the relative Kuo condition}\label{equivrel}

We first recall Rabier's function.
Let $ \mathcal{L}(E, F)$ denote the  space of linear  mappings from 
$\R^n$ to $\R^p$. 
For $T\in  \mathcal{L}(\R^n, \R^p),$  we denote by $\ T^*$ the adjoint map in 
$ \mathcal{L}(\R^p, \R^n).$

\begin{defn}
Let $T\in \mathcal{L}(\R^n, \R^p)$. Set
$$
\nu(T)=\displaystyle \inf\{\Vert T^{*}(v)\Vert:\ v\in \R^p,\ 
\Vert v\Vert=1\} ,
$$
where $T^*$ is the dual operator.
This function is called {\em Rabier's function} (\cite{R}).
\end{defn}
We have the following facts on Rabier's function $\nu(T)$
(see \cite{KOS}, \cite{R} and \cite{Jelonek} for instance):

\begin{enumerate}  
\item  Let $\mathcal{S}$ be the set of non surjective linear map in 
$ \mathcal{L}(\R^n,\R^p)$. 
Then for all $T\in \mathcal{L}(\R^n,\R^p)$, we have:
\begin{enumerate}  
\item $\dis
\nu(T)= d(T, \mathcal{S}):=\inf_{T'\in \mathcal{S}}\Vert T-T'\Vert.
$
\item 
$\dis
\nu(T)=\sup \{ r>0:  \, \  B(0,r)\subseteq T(B(0,r))\} .
$
\item If $T\in GL_{n}(\R)$, then 
$\dis \nu(T)=\frac{1}{\Vert T^{-1}\Vert}$.
\end{enumerate}  
\item  For $T,\ T'\in \mathcal{L}(\R^n,\R^p)$ we have  $ \nu(T+T')\geq \nu(T)-\Vert T'\Vert$.
\item  The relationship $\nu \thickapprox \kappa$ holds
between the Rabier's function and the Kuo distance.
More precisely, for $T=(T_{1}, \ldots,T_{p})\in \mathcal{L}(\R^n,\R^p)$, 
we have 
$$
\nu(T)\leq \kappa(T)\leq \sqrt{p}\ \nu({T}).
$$
\end{enumerate} 

\begin{defn}\label{minor}  Let $A=[a_{ij}]$ be the matrix in 
$\mathcal{M}_{n,p}(\R),\ n\geq p$. 
By $M_{I}(A)$,  we denote a $p\times p$ minor of $A$ indexed by $I$, 
where $I= (i_{1},\ \ldots,\ i_{p})$ is any subsequence of 
$(1, \ldots,\ n)$. 
Moreover, if $J= (j_{1},\ \ldots,\ j_{p-1})$ is any subsequence of $(1, \ldots, n)$ and $j\in\{1,\ \ldots,\ p\}$, then by $M_{J}(j)(A)$ we denote 
an $(p-1)\times(p-1)$ minor of a matrix given by columns indexed by $J$ 
and with deleted $j$-th row (if $p=1$ we put $M_{J}(j)(A)=1$). 
We define $\eta : \mathcal{M}_{n,p}(\R) \to \R_+$ by
$$
\eta (A) := \displaystyle \left(\frac{\sum_{I}\vert M_{I}(A)\vert^2}
{\sum_{J,j}\vert M_{J}(j)(A)\vert^2}\right)^{\frac{1}{2}}. 
$$
\end{defn} 

It is easy to see that the above $\eta$ is equivalent to
the following non-negative function $\tilde{\eta}$ in the sense that
$\eta \thickapprox \tilde{\eta}$:
$$
\tilde \eta (A):=\displaystyle \max_{I}\frac{|M_{I}(A)|}{h_{I}(A)},
$$
where $h_{I}(A)=\displaystyle \max\{|M_{J}(j)(A)|:J\subset I,
\ j=1,\ \ldots,\ p\}$,
with the  convention that   ${0\over 0}=0$.

\begin{lem}\label{equiv2}
The relationship $\eta \thickapprox \kappa$ holds.
More precisely, for $T=(T_{1}, \ldots,T_{p})\in \mathcal{L}(\R^n,\R^p)$, 
we have 
$$
\eta (T)\leq \kappa(T)\leq \sqrt{p}\ \eta ({T}).
$$
\end{lem}

\begin{rmk} The functions $\nu$, $\kappa$, $\eta$ and $\tilde \eta$ are 
continuous.
In order to see these facts, it suffices to show that $\eta$ is continuous at 
$A\in \mathcal{M}_{p,p}(\R)$. 
It is obvious if the denominator is bigger than $0$. 
Let us assume that the dominator is equal to $0$. 
Then $\eta (A)=0$ and thus $ A\in\mathcal{S}$. 
Let $\{ A_{k} \}$ be a sequence of elements of $\mathcal{M}_{p,p}(\R)$ 
which tend to $ A$. 
Then there exists $C>0$ such that as $ k\rightarrow\infty$, we have
$$
\eta (A_{k})\leq C\nu(A_{k})=C d(A_{k}, \mathcal{S})\rightarrow 0.
$$
\end{rmk}

We have the following equivalent conditions to the relative Kuo condition.
		
\begin{thm}\label{equiv1} Let $\s$ be a (non empty) germ  at $0$ of a closed 
subset of $ \R^n.$ For $f\in {\mathcal E}_{[r]}(n,p)$, $n \ge p$, 
the following conditions are equivalent:

\begin{enumerate}[(1)]
\item $f$ satisfies condition $(K_{\s})$.
\item $f$ satisfies condition $(\widetilde{K}_{\s})$.
\item The inequality
\begin{equation*}
d(x,\s)
  \left(\frac{\Gamma( (\grad f_i(x))_{1\leq i\leq p})}
{\sum_{j=1}^p\Gamma ((\grad f_i(x))_{i\ne j })}\right)^{\frac{1}{2}}
+\|f(x)\|\succsim d(x,\s)^{r} 
\end{equation*}\label{koz1}
holds in some neighbourhood of $0 \in \R^n$,
where 
$$
\Gamma (v_1,\ldots,v_k) :=\det (<v_i, v_j>_{\{ i,j \in\{1,\ldots, k\}\}})
$$ 
is the Gram determinant.
\item The inequality
\begin{equation*}
d(x,\s)\|df^*(x)y\|
+\|f(x)\|\succsim d(x,\s)^{r} 
\end{equation*} 
holds for $x$ in some neighbourhood of $0 \in \R^n$, uniformly for all 
$y \in \mathbb{S}^{p-1},$ 
where $df^*(x)$ is the dual map of $df(x)$,
and $\mathbb{S}^{p-1}$ denotes the unit sphere in $\R^p$ centred at 
$0 \in \R^p.$
\end{enumerate}
\end{thm} 

\begin{proof}
As mentioned above, it is easy to see the equivalence $(1) \iff (2)$.

The equivalence $(2) \iff (3)$ follows from Lemma \ref{equiv2} and 
$$
\D(f,x)  :=\sum_{1\leq i_1<\ldots<i_{p}\leq n} \left(\det {D (f_1,\ldots,f_p)\over D(x_{1_{1}},\ldots,x_{i_{p}})}\right)^2
=\Gamma ((\grad f_i(x))_{1\leq i\leq p}).
$$

Lastly we show the equivalence $(2) \iff (4)$. 
For a given set of vectors
$v_{1},v_{2},\dots,v_{p}\in\mathbb{R}^{n}$, let
$\tilde{\kappa}(v_{1},v_{2},\dots,v_{p})$ be defined as follows:
\begin{equation*}
\tilde{\kappa}(v_{1},v_{2},\dots,v_{p}) :=
\min\left\{ \left|\sum_{i=1}^{p}\lambda_{i}v_{i}\right| \ : \ \lambda_{i}\in\mathbb{R},\ \sum_{i=1}^{p}\lambda_{i}^{2}=1\right\}.
\end{equation*}
Then the equivalence between conditions (2) and (4) follows from the equivalence between $\kappa$ and $ \tilde{\kappa}$ (see  Lemma \ref{equiv2})  and the fact that
\begin{equation*}\|df^*(x)y\|=\|\sum_{i=1}^{p}y_{i}\grad f_{i}(x)\|\succsim d(x,\s)^{r-1} \  \text{  for all } y\in \mathbb{S}^{p-1}
\end{equation*}
is equivalent to
\begin{equation*} {\tilde\kappa}(\grad  f_{1}(x),\dots,  \grad f_{p}(x)) \succsim d(x,\s)^{r-1}.\end{equation*}
\end{proof} 

\subsection{Equivalent conditions to the second relative Kuo condition}
\label{equiv2ndrel}

For condition $(\widetilde{K}_{\s}^\d)$, we have the following 
equivalent conditions. 

\begin{prop}\label{equivdelta} 
Let $\s$ be a (non empty) germ  at $0$ of a closed subset of $ \R^n.$ For a map $f\in {\mathcal E}_{[r+1]}(n,p)$, $n \ge p$, the following conditions are equivalent:  
\begin{enumerate}[(1)]
\item $f$ satisfies condition $(\widetilde{K}_{\s}^\d)$.
\item For any  map $g\in {\mathcal E}_{[r+1]}(n,p)$ 
satisfying $j^{r}g(\s;0)=j^{r}f(\s;0)$, 
 there exists 
$\delta>0$  (depending on $g$), such that the inequality
\begin{equation*}
d(x,\s)
  \left(\frac{\Gamma( (\grad g_i(x))_{1\leq i\leq p})}
{\sum_{j=1}^p\Gamma ((\grad g_i(x))_{i\ne j })}\right)^{\frac{1}{2}}
+\|g(x)\| \succsim d(x,\s)^{r+1-\d} 
\end{equation*} 
holds in some neighbourhood of $0 \in \mathbb{R}^n.$
\item For any  map $g\in {\mathcal E}_{[r+1]}(n,p)$ 
satisfying $j^{r}g(\s;0)=j^{r}f(\s;0)$, 
there exists 
$\delta>0$  (depending on $g$), such that the inequality
\begin{equation*}
d(x,\s)\|dg^*(x)y\|
+\|g(x)\|\succsim d(x,\s)^{r+1-\d} 
\end{equation*} 
holds for some $\d>0$, for $x$ in some neighbourhood of $0 \in \R^n$ and uniformly for all $y\in \mathbb{S}^{p-1}.$ 

\item For any  map $g\in {\mathcal E}_{[r+1]}(n,p)$ 
satisfying $j^{r}g(\s;0)=j^{r}f(\s;0)$, there exists 
$\delta>0$  (depending on $g$), such that the inequality
\begin{equation}\label{koz2}
d(x,\s)\|df^*(x)y\|
+\|g(x)\|\succsim d(x,\s)^{r+1-\d} 
\end{equation} 
holds in a neighbourhood of the origin, (uniformly) for all 
$y\in \mathbb{S}^{p-1}.$
\end{enumerate}
\end{prop} 

\begin{proof}
The proofs of the equivalences between (1) (2),  and (3) are identical 
to the corresponding ones in Theorem \ref{equiv1}.

The equivalence between $(3)$ and $(4)$ is an application of 
Lemma\ \ref{lemrflat}.
\end{proof} 

In the next theorem, we establish the equivalence between the second 
relative Kuo condition ($K_{\s}^\delta$) and condition
$(\widetilde{K}_{\s}^\de)$  for subanalytic maps and $\s.$

\begin{thm}\label{prop223} 
Let $\s$ be a germ of a closed  subanalytic set at $0 \in \R^n$,
such that $0$ is an accumulation point of $ \s$ and $\mathbb{R}^n\setminus \s.$\\
For a subanalytic map $f\in {\mathcal E}_{[r+1]}(n,p)$, $n \ge p$, 
the following conditions are equivalent:  
\begin{enumerate}[(1)]
\item $f$ satisfies the second relative Kuo condition $(K_{\s}^\d):$ 
for any subanalytic map $g\in {\mathcal E}_{[r+1]}(n,p)$ 
with $j^{r}g(\s;0)=j^{r}f(\s;0)$, there are strictly positive numbers
$C, \alpha,\delta$ and $ \bar w$ (depending on $g$) such
that
\begin{equation}\label{ksd}
\kappa(df(x))\geq Cd(x,\s)^{r-\delta} \text{ in } \mathcal{H}^{\s}_{r+1}(g;\bar w)\cap\{\|x\|<\a\}.
\end{equation}

\item For any subanalytic
map $g\in {\mathcal E}_{[r+1]}(n,p)$ 
with $j^{r}g(\s;0)=j^{r}f(\s;0)$,
\begin{equation}\label{itksd}
 \frac{d(x,\s)\kappa(df(x))
+\|g(x)\|}{d(x,\s)^{r+1} }\to \infty\quad \text{ as } d(x, \s)\to 0,\, d(x, \s)\not=0.
\end{equation} 

\item $f$ satisfies condition $(\widetilde{K}_{\s}^\de):$
for any subanalytic
map $g\in {\mathcal E}_{[r+1]}(n,p)$ 
with $j^{r}g(\s;0)=j^{r}f(\s;0)$, there exists 
$\delta>0$  (depending on $g$) such that
\begin{equation}\label{tksd}
d(x,\s)\kappa(df(x))+\|g(x)\|\succsim d(x,\s)^{r+1-\d} 
\end{equation} 
holds in some neighbourhood of $0 \in \R^n.$

\end{enumerate}
\end{thm} 

\begin{proof}
We first show the implication $(1)\implies (2).$
If condition $(2)$ does not hold, then there exist a realisation $g$ of 
$j^{r}f(\s;0)$, and an analytic arc $\gamma:I\to \mathbb{R}^n$ 
where $I = [0, \beta )$, $\beta > 0$, such that 
$\gamma (0) = 0 \in \R^n$ and for $t\in I$
\begin{equation}\label{Eq42}
    \|g(\gamma(t))\| \precsim d(\gamma(t),\s)^{r+1},\quad
    \kappa(df(\gamma(t)))\precsim  d(\gamma(t),\s)^{r}.
\end{equation}
In particular, for any sequence $\{x_{i}\}$ where
$x_{i}=\gamma(t_{i})$, $t_{i}\to0$, $t_{i}\neq0$, and for any  $\de>0$ we have 
\begin{equation}\label{Eq50}
    \|g(x_{i})\|=o( d(x_{i},\s)^{r+1-\d}),\quad
    \kappa(df(x_{i}))=o( d(x_{i},\s)^{r-\d}).
\end{equation}
Then (\ref{Eq50}) implies that for any choice of the positive numbers $C, \alpha,\delta$ and $ \bar w$, 
the inequality
\begin{equation*}
\kappa(df(x))\geq Cd(x,\s)^{r-\delta} 
\end{equation*} 
cannot hold in $\mathcal{H}^{\s}_{r+1}(g;\bar w)\cap\{\|x\|<\a\}$, namely 
condition $(K_{\s}^\d )$ is not satisfied.
Therefore the implication $(1) \implies (2)$ is shown.

We next show the implication $(2)\implies (3).$
By condition $(2),$ the  continuous subanalytic function germ on 
$(\mathbb{R}^n\setminus \s,0)$, defined by
\begin{equation*} 
\dis h(x):= \frac{d(x,\s)^{r+1}}{ d(x,\s)\kappa(df(x))+\|g(x)\|},
\end{equation*} 
can be extended continuously by $0$ on $ \s.$
Since $ \s= h^{-1}(0) $,  by the Lojasiewicz inequality 
(\cite{lojasiewicz} \S 18), there is some $\d>0$ such that
$$
0\le h(x)\precsim d(x, \s)^\d
$$
in some neighbourhood of $0\in \mathbb{R}^n.$
Thus  $(\widetilde{K}_{\s}^\de)$ is satisfied.

We lastly show the implication $(3)\implies (1).$
Suppose that $f$ satisfies condition ($\widetilde{K}_{\s}^\d$).
Let $g \in {\mathcal E}_{[r+1]}(n,p)$ with 
$j^{r}g(\s;0)=j^{r}f(\s;0)$ which satisfies the condition that
there are positive constants $ \delta,C$ and  $\alpha$ such that
\begin{equation}\label{equ2.40}
d(x,\s)\kappa(df(x))
+\|g(x)\|\ge Cd(x,\s)^{r+1-\delta} 
\end{equation}
for $x\in\mathbb{R}^{n}$, $\|x\|<\alpha.$
If $x$ is in  the horn-neighbourhood 
$\mathcal{H}^{\s}_{r+1}(g;\frac{C}{2})\cap \{\|x\|<\alpha\},$ then
$$\|g(x)\|\leq\frac{C}{2}d(x,\s)^{r+1}\le\frac{C}{2}d(x,\s)^{r+1-\delta}
$$
and by \eqref{equ2.40}
$\dis
\kappa(df(x))\ge
\frac{C}{2}d(x,\s)^{r-\delta}.
$
Therefore condition $(K_{\s}^\d)$ is satisfied;
which shows the implication $(3) \implies (1)$.
\end{proof} 

\begin{rmk}\label{nonsubanalytic}
From the proof, we can see that without the assumption of subanalyticity,
the implication $(3) \implies (1)$ in Theorem \ref{prop223} holds, 
namely condition $(\widetilde{K}_{\s}^\de)$ implies the second relative 
Kuo condition $(K_{\s}^\d).$ 
\end{rmk}

Let us now introduce  a (ostensibly) weaker condition in terms of  $f$ only 
(namely,  not using all the realisations of the jet  $j^rf(\s;0)$), to be 
compared to V. Kozyakin \cite{kozyakin}.

\begin{defn} A map germ $f\in {\mathcal E}_{[r+1]}(n,p)$, $n\geq p$, satisfies 
condition {\it  (${KZ}_{\s}$)}
 if
\begin{equation}
 \frac{d(x,\s)\|df^*(x)y\|
+\|f(x)\|}{d(x,\s)^{r+1} }\to \infty\quad \text{ as } d(x,\s)\to 0,\, d(x,\s)\not=0
\end{equation} 
in a neighbourhood of the origin, (uniformly) for all 
$y \in \mathbb{S}^{p-1}$. 
\end{defn} 

We have another equivalent condition to the second relative Kuo condition. 

\begin{thm}\label{prop226}
Let $\s$ be a germ at $0$ of a closed subanalytic subset of $ \mathbb{R}^n,$ 
and let $f\in {\mathcal E}_{[r+1]}(n,p),$ $n\geq p,$ be a subanalytic map.
Then f satisfies the second relative Kuo condition, equivalently 
for any subanalytic map $g\in {\mathcal E}_{[r+1]}(n,p)$ 
such that $j^{r}g(\s;0)=j^{r}f(\s;0)$, there are positive constants
$\delta,C$ and  $\alpha$ such that
\begin{equation}\label{equ2.4}
d(x,\s)\kappa(dg(x))
+\|g(x)\|\ge Cd(x,\s)^{r+1-\delta},
\end{equation}
for $\| x\| < \alpha$
if and only if $f$ satisfies condition (${KZ}_{\s}$).
\end{thm} 

\begin{proof} 
By Theorem \ref{equivdelta}, condition (\ref{equ2.4}) is equivalent to condition  $(4)$,
which implies (${KZ}_{\s}$) for any $g\in {\mathcal E}_{[r+1]}(n,p)$ such that $j^{r}g(\s;0)=j^{r}f(\s;0)$, in particular, for $f.$

To prove the converse, we will use the subanalyticity. 
Suppose that condition (${KZ}_{\s}$) is satisfied. 
Let $g\in {\mathcal E}_{[r+1]}(n,p)$ such that $j^{r}g(\s;0)=j^{r}f(\s;0),$
and set $h(x) :=g(x)-f(x)$. 
By Lemma\ \ref{lemrflat}, $\|h(x)\|\precsim
d(x,\s)^{r+1}$ for sufficiently small values of $\|x\|.$
Then, for all $y\in \mathbb{S}^{p-1}$
$$
\dis d(x,\s)\|df^{*}(x)y\|+ \|g(x)\|=d(x,\s)\|df^{*}(x)y\|+ \|f(x)+h(x)\|,
$$
and
$$
\dis\frac{d(x,\s)\|df^{*}(x)y\|+ \|g(x)\|}{d(x,\s)^{r+1}}\ge
\frac{d(x,\s)\|df^{*}(x)y\|+ \|f(x)\|}{d(x,\s)^{r+1}}-
\frac{\|h(x)\|}{d(x,\s)^{r+1}}.
$$
Since $\dis\frac{\|h(x)\|}{d(x,\s)^{r+1}}$ is bounded,
$
\dis\lim_{d(x,\s)\to 0} \frac{d(x,\s)\|df^{*}(x)y\|+ \|f(x)\|}{d(x,\s)^{r+1}}=\infty
$
which implies
$
\dis\lim_{d(x,\s)\to 0}\frac{d(x,\s)\|df^{*}(x)y\|+ 
\|g(x)\|}{d(x,\s)^{r+1}}=\infty .
$
Therefore condition (\ref{koz2}) is satisfied.

Now the  continuous subanalytic function germ on 
$(\mathbb{R}^n\setminus \s,0)\times \mathbb{S}^{p-1}$, defined by
\begin{equation*} 
\dis q(x,y):= \frac{d(x,\s)^{r+1}}{ d(x,\s)\|df^*(x)y\|+\|g(x)\|},
\end{equation*} 
can be extended continuously  by $0$ to $ \s\times \mathbb{S}^{p-1}.$
Since $ \s\times \mathbb{S}^{p-1}= q^{-1}(0) $, by the Lojasiewicz inequality, 
there is some $\d>0$ such that
$$
|q(x,y)|\precsim d((x,y), \s\times \mathbb{S}^{p-1})^\d=d(x, \s)^\d
$$
in some neighbourhood of $0\in \mathbb{R}^n.$
This is exactly condition $(4)$ in Theorem \ref{equivdelta}.
\end{proof}


\bigskip


\end{document}